\documentclass[journal,twoside,print,9pt]{ieeecolor}
\usepackage{generic}
\usepackage{verbatim}
\usepackage[export]{adjustbox}
\usepackage{amsmath,amssymb}

\usepackage{etoolbox}
\makeatletter
\@ifundefined{color@begingroup}%
  {\let\color@begingroup\relax
   \let\color@endgroup\relax}{}%
\def\fix@ieeecolor@hbox#1{%
  \hbox{\color@begingroup#1\color@endgroup}}
\patchcmd\@makecaption{\hbox}{\fix@ieeecolor@hbox}{}{\FAILED}
\patchcmd\@makecaption{\hbox}{\fix@ieeecolor@hbox}{}{\FAILED}
\def\BibTeX{{\rm B\kern-.05em{\sc i\kern-.025em b}\kern-.08em
    T\kern-.1667em\lower.7ex\hbox{E}\kern-.125emX}}
\usepackage{graphicx, subfig, booktabs}
\usepackage{amsmath,amssymb,amsthm, bm}
\usepackage{cases}
\usepackage[T1]{fontenc}
\usepackage{hyperref}
\allowdisplaybreaks
\hypersetup{
     colorlinks =true,
     linkcolor = blue,
     filecolor = blue,
     citecolor = magenta,      
     urlcolor = magenta,
     }
\usepackage{setspace}
\usepackage{siunitx}
\usepackage{csvsimple}
\usepackage{xurl}
\usepackage[scaled]{helvet}
\usepackage[T1]{fontenc}

\usepackage[style=ieee,doi=true]{biblatex}
\AtBeginBibliography{\footnotesize}
\let\labelindent\relax
\usepackage{enumitem, xspace}

\usepackage{accents}

\newcommand{\NR}{Newton-Raphson\xspace}
\newcommand{\NF}{\mathcal{N}\mathcal{F}}

\theoremstyle{remark}
\newtheorem{remark}{Remark}
\newtheorem{definition}{Definition}

\newtheorem{proposition}{Proposition}
\newtheorem{lemma}{Lemma}
\newtheorem{theorem}{Theorem}
\newtheorem{corollary}{Corollary}

\newcommand{\edit}[1]{\textcolor{black}{#1}}
\begin{document}

% LA-UR-24-31207
\title{Hierarchical Network Partitioning for Solution \\of Potential-Driven, Steady-State Nonlinear \\ Network Flow Equations}

\author{Shriram Srinivasan$^{\dagger}$, Kaarthik Sundar$^{*}$
\thanks{$^{\dagger}$Applied Mathematics and Plasma Physics Group, Los Alamos National
Laboratory, Los Alamos, New Mexico, USA. E-mail: \texttt{shrirams@lanl.gov}}\;
\thanks{$^{*}$Information Systems and Modeling Group, Los Alamos National
Laboratory, Los Alamos, New Mexico, USA. E-mail: \texttt{{kaarthik}@lanl.gov}}\;
\thanks{The authors acknowledge the funding provided by LANL’s Directed Research and Development (LDRD) project. The research work conducted at Los Alamos National Laboratory is done under the auspices of the National Nuclear Security Administration of the U.S. Department of Energy under Contract No. 89233218CNA000001.}\; 
}

\maketitle

\begin{abstract}
The solution of potential-driven steady-state flow in large networks is a task which manifests in various engineering applications, such as transport of natural gas or water through pipeline networks.
The resultant system of nonlinear equations depends on the network topology and in general there is no numerical algorithm that offers  guaranteed convergence to the solution (assuming a solution exists). 
Some methods offer guarantees in cases where the network topology satisfies certain assumptions, but these methods fail for larger networks.
On the other hand, the Newton-Raphson algorithm offers a convergence guarantee if the starting point lies close to the (unknown) solution. 
It would be advantageous to compute the solution of the large nonlinear system through the solution of smaller nonlinear sub-systems wherein the solution algorithms (Newton-Raphson or otherwise) are more likely to succeed. 
This article proposes and describes such a procedure, an hierarchical network partitioning algorithm that enables the solution of large nonlinear systems corresponding to potential-driven steady-state network flow  equations.
\end{abstract}

\begin{IEEEkeywords}
steady-state, network flow equations, potential, \NR, partition, hierarchical
\end{IEEEkeywords}
\IEEEpeerreviewmaketitle

\def\d{\partial}

\section{Introduction}
\label{sec:intro}

Potential-driven steady-state flow in networks is an abstract problem with  manifestations in various engineering applications, such as those of natural gas \cite{singh2019optimal}, water \cite{Singh2020-water}, electric power transmission \cite{purchala2005usefulness} or in fractured porous media modelled as discrete fracture networks (DFNs) \cite{ss_flux_thresholding2018}. The task in this problem is to find nodal potentials and  flows through different components representing the network edges given the network topology, nodal injections and other data about  the components that make up the network.
In general, while the problem is simple when restricted to a single edge of a network, it ceases to be so when  a large network  is considered especially if one considers a nonlinear relationship between the nodal potentials and the flow on the edge connecting them.
The resultant system of nonlinear equations depends on the network topology and in general there is no numerical algorithm that offers  guaranteed convergence to the solution (assuming a solution exists).

The Newton-Raphson (NR) algorithm is the most popular method for solving systems of non-linear equations, but convergence is guaranteed only if the initial guess is within a certain basin of attraction centered at the solution \cite{Ciarlet2012Jul}. In other words, the NR method is sensitive to the starting point and needs a good initial guess. As the system size and number of dimensions increase, it becomes challenging to design a good starting point, especially since the basin of attraction   can show fractal behaviour \cite{kovacs2011understanding}.
In \cite{nrsolver}, the authors alleviated the situation by scaling the variables in the network flow equations so that all the transformed variables had the same order of magnitude, and thus the method performed well in numerical experiments with small to moderate network sizes. However, scaling offers no convergence guarantees and does not change the essential feature of the algorithm so that the search for a solution continues to get more challenging for larger networks. Moreover, it is not clear how to determine an appropriate initialization when parameters move outside typical operating ranges as may  occur in contingency analysis.

In order to circumvent these issues, various approaches distinct from NR have been tried that can guarantee convergence to a solution from any starting point, but they have limited applicability due to restrictive hypothesis or inherent scaling issues that render them useless for large networks. 
In the context of natural gas as well as water distribution networks, there have been attempts to use Mixed-Integer Quadratically Constrained Quadratic Program (MIQCQP) \cite{Singh2020,Singh2020-water}, Mixed-Integer Second Order Cone Program (MISOCP) \cite{Singh2019}, and Semi-Definite Program (SDP) approaches \cite{ojha2017solving} which are all computationally expensive and intractable for large instances.
\edit{The MIQCQP approach \cite{Singh2020,Singh2020-water}, for example, is an NP-hard problem which cannot scale well for large instances despite restrictions such as no overlapping cycles and no circulatory flows in cycles.}
Similarly, \cite{de2000gas} proposes a primal-dual based method to solve the network flow problem, but only for  networks without compressors while \cite{Dvijotham2015Jun} aims to exploit monotone operator theory to find a solution but requires an expensive calculation  to find pre-multiplication matrices that can render the operator monotone.

It would be advantageous if the solution of the large nonlinear system could be computed through  the solution of  smaller nonlinear sub-systems  wherein the NR algorithm or any of the aforementioned methods are more likely to succeed. 
\edit{With similar motivation,  in \cite{Rios-Mercado2002Nov}, the authors split the network by removing compressors and create hyper-nodes, but then make unsubstantiated and non-trivial assumptions on the topology of the hyper-nodes.}

In this article, we present a rigorous technique that relies on minimal assumptions about the network structure to partition the network and solve the full nonlinear system by hierarchical traversal through the partition, solving smaller nonlinear sub-systems in each instance.
Moreover, the algorithm is presented for an abstract network flow problem so that it will be immediately clear that it is applicable for a wide array of applications that includes natural gas and water distribution as special cases.

\section{Governing equations}
\label{sec:model-equations}

In order to describe the potential-driven steady-state flow equations on a network in general terms, we first introduce some notation. 
Let $G$ denote the graph of the network, where $V(G)$ and $E(G)$ are the corresponding set of \emph{junctions/nodes} and \emph{edges} of the network. 
% The edge set is given by $E = P \bigcup C$ where $P$ and $C$ are the set of \emph{passive} and \emph{control} elements respectively.  
% The junctions $i$ and $j$ are connected by an element $(i, j)$ that belongs to either $P$ or $C$.   
For each junction $i \in V(G)$, we let $\pi_i \in \mathbb{R}$ denote the nodal value of the potential while $q_i \in \mathbb{R}$ is the nodal injection or supply.  
We let $f_{ij} \in \mathbb{R}$ denote the steady flow through the edge  $(i, j) \in E(G)$  that connects junctions $i, j \in V(G)$.

The system of network flow equations  consisting of balance of flows at each junction and  potential-driven flow on each edge  has the general form ($*$ superscript indicates given data)
\begin{subnumcases} {\label{eq:NF}
\mathcal N \mathcal F:}  
\gamma^{*}_{ij}\pi_i - \pi_j = g^{*}_{ij}(f_{ij}) \quad \forall (i, j) \in E(G) , &  \label{eq:NF-edge}\\ 
% \gamma^{*}_{ij} \pi_i - \pi_j = 0 \quad \forall (i, j) \in C, & \label{eq:NF-compressor} \\ 
\sum_{(i, j) \in E(G)} f_{ij}  - \sum_{(j, i) \in E(G)} f_{ji}  = q^{*}_i \; \forall i \in V(G). & \label{eq:NF-balance}
\end{subnumcases}
In this general form, $g^{*}_{ij}: \mathbb{R} \mapsto \mathbb{R}$ is a  monotonic function, $q^{*}_i \in \mathbb{R}$ while  $\gamma^{*}_{ij} \in \mathbb{R}^{+}$.
Various choices and specification of $g^{*}_{ij}, \gamma^{*}_{ij}$ correspond to different types of active or passive edge elements in network flow applications.

There is an inherent non-uniqueness in the nodal potentials that appear in \eqref{eq:NF-edge}, since the equation could be invariant to certain additive perturbations of the nodal potentials. 
Specifically, the set of potentials satisfying \eqref{eq:NF-edge}  is unique if and only if the network $G$ has at least one cycle for which some edge $(k,l) \in E(G)$ in the cycle  has $\gamma^{*}_{kl} \neq 1$.
\edit{In order to eliminate the non-uniqueness for all network topologies, it is sufficient (but not always necessary) that the potential be known or fixed  at certain junctions.} 
% note to remember - if you dont have slack node, then all q given must be  known and should satisfy sum zero. 
% For a network with compressor in loop, can solve system as usual without slack
% For network with no compressor in loop: need slack
% For a tree - need slack for pressure uniqueness
% For a tree (zero, 1 or more slacks) Af = q  has unique solution upto 1 slack and infinite number of solutions for 2 or more slacks.
Hence each junction $i \in V(G)$  is assigned to one of two mutually disjoint subsets of $V(G)$ consisting of so-called slack junctions $N_s$ and non-slack junctions $N_{ns}$ respectively, depending on whether the potential $\pi_i^*$ or injection $q_i^*$ at the node is specified. 
Thus $V = N_s \cup N_{ns}$ and one may modify the system $\NF$ \eqref{eq:NF} to:
\begin{subnumcases} {\label{eq:NF-pi}
\mathcal N \mathcal F_{\pi}:}  
\gamma^{*}_{ij}\pi_i - \pi_j = g^{*}_{ij}(f_{ij}) \quad \forall (i, j) \in E(G), &  \label{eq:NF-pi-edge}\\ 
\sum_{(i, j) \in E(G)} f_{ij}  - \sum_{(j, i) \in E(G)} f_{ji}  = q^{*}_i \; \forall i \in N_{ns}, & \label{eq:NF-pi-balance}\\
\pi_i = \pi^{*}_i \; \forall i \in N_s. & \label{eq:NF-pi-slack}
\end{subnumcases}

The non-uniqueness of the potentials is no longer an issue due to \eqref{eq:NF-pi-slack} and we are now in a position to relate the system $\NF_{\pi}$ \eqref{eq:NF-pi}  in $|V(G)| + |E(G)|$ variables to particular instances of network flow relevant to various domains.

In a single-phase electric transmission network, utilizing a Direct Current (DC) approximation for the physics of electric power flow \cite{purchala2005usefulness}, the nodal quantities are \emph{voltage phase angles} that drive \emph{active power flow}, and the edge elements are transmission lines where $\gamma^{*}_{ij} =1$ and the resistance function $g^{*}_{ij}(f_{ij}) \propto f_{ij}$, i.e., linear in $f_{ij}$. 
Thus, the edge relation \eqref{eq:NF-pi-edge} corresponds to DC power flow, \eqref{eq:NF-pi-balance} to balance of power at the nodes, and \eqref{eq:NF-pi-slack} to junctions with specified voltage phase angles (also called ``reference'' junctions in the power grid \cite{purchala2005usefulness}).
Flow in DFNs offers an  identical analogue \cite{Karra2018Mar} if we term the nodal quantities as \emph{pressures} governing \emph{fluid flows}, \eqref{eq:NF-pi-balance} as balance of mass and the edge relation \eqref{eq:NF-pi-edge} as Darcy's law for porous media.
\edit{In both of these cases, the system is usually presented in an equivalent form obtained by exploiting the linearity of $g^{*}_{ij}$ to eliminate the edge flows whereupon the Laplacian matrix of the network is seen to define a \emph{linear} system in the nodal quantities \cite{Karra2018Mar} that poses minimal challenge to numerical solution techniques.}

The flow of fluid in pipeline networks corresponds to the case where edge elements can be modelled as pipes by setting  $\gamma^{*}_{ij} =1$ and the resistance function $g^{*}_{ij}(f_{ij}) \propto f_{ij} |f_{ij}|$ \cite{nrsolver} is nonlinear. 
For compressible fluids such as natural gas or hydrogen \cite{kazi2024modeling}, the choice $g^{*}_{ij} = 0$ corresponds to pressure regulators or compressors depending  on whether $\gamma^{*}_{ij}$ is  less than or equal to unity \cite{Schmidt2017}.
For incompressible liquids such as water, pumps can be modelled as edges where 
$\gamma^{*}_{ij} =1$ and $g^{*}_{ij}$ is some positive constant or a monotonic function of $f_{ij}$ \cite{singh2019optimal}. 
%\cite{tasseff2022polyhedral}, 

For any \emph{nonlinear} edge relation, including the aforementioned case of fluid flow in pipeline networks,  \eqref{eq:NF-pi} is a nonlinear system of algebraic equations in $|V(G)| + |E(G)|$ variables that in general becomes more challenging to solve for larger networks.
For the specific case of fluid flow in pipeline networks, the existence of a solution   has been shown in \cite{ss-soln-existence} using topological degree theory. 
We shall make no attempt to follow the same line of argument to examine existence of a solution to \eqref{eq:NF-pi}, but we will in passing state  a result on uniqueness of solution to \eqref{eq:NF-pi}  when certain assumptions on the network topology are valid.  
These assumptions are not restrictive, serving instead to eliminate certain impractical but pathological cases, and are stated as follows \cite{nrsolver}:
\begin{enumerate}[label=(A\arabic*)]
    \item There is at least one slack junction, i.e., $|N_s| \geqslant 1$.\label{assumption:slacks}
    \item When  $|N_s| \geqslant 2$, a path connecting two slack junctions must consist of at least one edge with $g^{*}_{ij} \neq 0$ \label{assumption:path-pipe}
    \item Any cycle must consist of least one edge with $g^{*}_{ij} \neq 0$. \label{assumption:cycle-pipe}
\end{enumerate}
\ref{assumption:slacks} ensures that the equation \eqref{eq:NF-pi} is not invariant to additive perturbations of the potential. 
\ref{assumption:path-pipe} and \ref{assumption:cycle-pipe} are required (see \cite{nrsolver, Singh2019, Singh2020} for details) in order to ensure uniqueness of the flows satisfying \eqref{eq:NF-pi}.

\begin{proposition}
\label{prop:uniqueness}
    Under the condition that \ref{assumption:slacks}, \ref{assumption:path-pipe}, and \ref{assumption:cycle-pipe} hold, if the system \eqref{eq:NF-pi} has a solution, it must be unique. 
\end{proposition}

\begin{proof}
    The proof follows from the same arguments that underpin  \cite{nrsolver}, namely, that the edge relation $g^{*}_{ij}$ is a monotonic function of $f_{ij}$ and that that \ref{assumption:slacks}, \ref{assumption:path-pipe}, and \ref{assumption:cycle-pipe} hold.
\end{proof}

\section{Method outline}
\label{sec:method-outline}

In this article,  we propose a partitioning algorithm that allows us to compute the solution to  \eqref{eq:NF-pi} for a large network  via the solution of \eqref{eq:NF-pi} on the partitioned sub-networks when the size of the full system presents a challenge to  the standard solution algorithm.
\emph{At all times, we shall assume that a solution exists.}
In order to motivate the subsequent mathematical description of the theoretical underpinnings of the algorithm, we first present  an intuitive, schematic description of it in Figures~\ref{fig:algo-pic}~and~\ref{fig:chunks}.

The context for the algorithm is the solution of \eqref{eq:NF-pi} specialized for flow of natural gas with the topology implied by GasLib-40 \cite{Schmidt2017} whose graph is depicted in Figure~\ref{fig:original-network} with the slack node coloured green.
From the visual depiction of the network, one can recognize that the node coloured orange is centrally located at the intersection of  several  branches.

If we imagine removing the orange coloured node, we see that it produces four connected sub-networks. Hence we replicate the node in each connected sub-network to produce the modified network in Figure~\ref{fig:duplicated-vertices}. 
A replicated node has the same potential as the original node.

In the modified network, the shaded areas in Figure~\ref{fig:chunks} delineate sub-networks, which if viewed as vertices, form a  tree with one slack node.

This tree structure then makes the subsequent steps clear: \\
(1) Solve a flow problem on the tree with injections given by the sum of all the injections of nodes that belong to the sub-network vertex. The solution then  determines the transfer or injection into the slack network $N_1$. \\
(2) Now solve \eqref{eq:NF-pi} for $N_1$. \\
(3) The computed potential for the replicated node in $N_1$ is then transmitted to the other replicated nodes in $N_2, N_3,N_4$. \\
(4) Now solve \eqref{eq:NF-pi} for $N_2, N_3, N_4$. \\
(5) The potentials and flows in vertices and   edges of Figure~\ref{fig:duplicated-vertices} constitute our solution to \eqref{eq:NF-pi} on the network depicted in Figure~\ref{fig:original-network}.

Our explanation of this procedure raises several questions: \\
(Q1) How can we identify vertices for replication without the benefit of a visual representation?
(Q2) Can we choose multiple vertices for replication? \\
(Q3) Will the reduced-graph (where sub-networks play the role of vertices) always be a tree? \\
(Q4) Would it always be possible to start from one sub-network (like $N_1$ here) and obtain the solution to system~\eqref{eq:NF-pi} for the full network? 

All of the above will be answered in the affirmative  with the mathematical development of the theory described in the next section. 
However, anticipating the developments, we can state that \emph{cut-points} or \emph{articulation points} of a network are the key concepts that ensure the \emph{block-cut tree} representation.

\begin{figure}[htb]
\centering
\subfloat[]%[Original network]
{\includegraphics[scale=0.9]{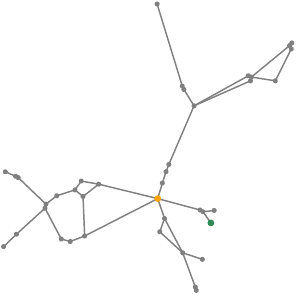}\label{fig:original-network}}\\
\subfloat[]%[Equivalent network with duplicated vertices
{\includegraphics[scale=0.9]{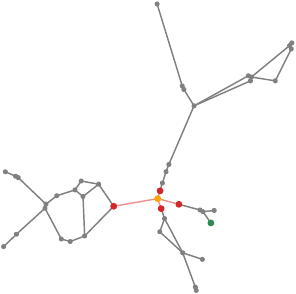}\label{fig:duplicated-vertices}}
\caption{The slack vertex is coloured green while the replicated vertices are shown in red. The chosen vertex (coloured orange) of the original network in (a) is replicated as shown to create a new, equivalent network in (b).}
\label{fig:algo-pic}
\end{figure}

\begin{figure}
    \centering
    \includegraphics[scale=.9]{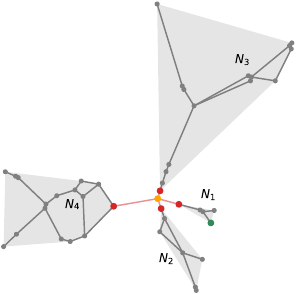} ~~
    \includegraphics[fbox, scale=0.9]{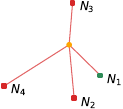}
\caption{We can visually discern  four sub-networks (blocks shown in blue colour) connected to each other (through the articulation point) in the form of a tree (shown inset, which is the block-cut tree).  Using the cumulative injections, can determine flow in each edge of the tree. If we then solve the system \eqref{eq:NF-pi} for slack sub-network $N_1$, can subsequently solve $N_2$, $N_3$ and $N_4$ to determine solution of system \eqref{eq:NF-pi} for the  full network.}
 \label{fig:chunks}
\end{figure}

%----------------------------
\section{Mathematical description of the algorithm} \label{sec:background}

\subsection{Block-Cut Trees}
Let $G$ be the graph of the network where the edge set  and vertex set are denoted by $E(G)$ and $V(G)$ respectively. We denote the number of vertices and edges by $|V(G)|$ and $|E(G)|$ respectively.
% and assume that $|V(G)| \geqslant 3$.
We recall the following  standard definitions (\cite{harary-book,harary}):

\begin{definition}
\label{def:articulation-pt}
A \emph{cut-point} or \emph{articulation point} of a connected graph is a vertex whose removal increases the number of connected components.
\end{definition}

\begin{definition}
\label{def:bi-connected}
A \emph{non-separable} or \emph{bi-connected} graph has no articulation points. The simplest non-separable graph is a complete graph with three vertices, but we shall also treat a graph with two vertices as non-separable.
\end{definition}

\begin{definition}
\label{def:block}
A  \emph{block} is a maximal connected sub-graph that is \emph{non-separable}. 
% The case $|V(G)| = 2$ is also assumed to be a \emph{block}.
\end{definition}

Corresponding to any $G$, if $\mathcal{C}$ be the set of cut-points and $\mathcal{B}$ the set of blocks of $G$, one can associate a (bi-partite) \emph{block-cut tree} $\operatorname{bc}(G)$ as follows:
\begin{definition}
\label{def:block-cut-tree}
\edit{The block cut tree of $G$ is a graph whose vertex set comprises  members of two distinct sets -- the set of articulation points $\mathcal C$, and the set of blocks $\mathcal B$ where each block $B \in \mathcal B$ is thought of as a vertex, i.e.,
$V(\operatorname{bc}(G)) = \mathcal{B} \cup \mathcal{C}$. In $\operatorname{bc}(G)$ there is an edge connecting $B \in \mathcal{B}$ and $c \in \mathcal{C}$ whenever $c \in V(B)$.}
\end{definition}
\edit{In the above definition, we have followed the same notation as Harary \cite{harary,harary-book} who introduced the notion of block and block cut-tree in graph theory.}
\edit{Observe that, we have 
% $\bigcup_{B \in \mathcal B} E(B) = E(G)$ and for any $B_i, B_j \in \mathcal B$, if $B_i \cap B_j \neq \emptyset$, then $B_i \cap B_j = c$ for some $c \in \mathcal C$.
\begin{subequations}
    \begin{gather}
     \underset{B \in \mathcal{B}}{\bigcup} E(B)  = E(G), \\
    \text{and for $B_i, B_j \in \mathcal B$  with $V(B_i) \cap V(B_j) \neq \emptyset$,}  \notag  \\
      V(B_i) \cap V(B_j)  = \{c\}  \; \text{for  some} \;  c \in \mathcal C 
    \end{gather}
    \label{eq:generalized-blocks}
\end{subequations}
}
\begin{remark}
\edit{Another representation, called a \emph{block graph} has vertex set $\mathcal{B}$ and an edge between $B_i, B_j \in \mathcal{B}$ whenever $V(B_i) \cap V(B_j) \neq \emptyset$. However, the block graph need not be a tree, and hence it is not of interest in this article. }
\end{remark}
The block-cut tree $\operatorname{bc}(G)$ is a simplified representation of $G$ wherein sub-networks of $G$ are condensed into nodes. 
Without this simplification, it is an \emph{augmented graph} $G_{\mathcal{BC}}$  defined as follows:
\begin{definition}
\label{def:topology-change-replication}
\edit{The \emph{augmented graph} $G_{\mathcal{BC}}$  has the vertex set $V(G_{\mathcal{BC}}) = \bigcup_{B \in \mathcal{B}} V(B) ~\cup~  \mathcal{C}$ and 
differs from $G$ in possessing replicas of articulation points in each block, and  having  edges connecting each of the replicated vertices with the corresponding original vertex. }
\end{definition}

\subsection{Generalized Block-Cut Trees}
The size of the blocks in the block-cut tree is fixed and determined by the network topology. In many situations, however, it would be advantageous if one could have an analogous tree representation that could start with larger ``blocks" as vertices that successively  get refined to eventually yield the standard block-cut tree.

In order to do so, we consider a set $\mathcal{B}$ of \emph{generalized blocks} 
that adhere to the conditions stated in Eq. \eqref{eq:generalized-blocks}  but need not be non-separable as in Definition~\ref{def:block}.

\begin{lemma}
\label{lemma:generalized-block-cut-tree}
\edit{Consider set $\mathcal{B}$ consisting of selected connected sub-graphs (generalized-blocks) of $G$, and $\mathcal{C} \subset V(G)$ that satisfy Eq~\eqref{eq:generalized-blocks}. A bipartite graph can be constructed with vertex set $\mathcal{B} \cup \mathcal{C}$ with an
edge connecting $B \in \mathcal{B}$ and $c \in \mathcal{C}$ whenever $c \in V(B)$. Such a graph $\operatorname{bc}(G; \mathcal{B}, \mathcal{C})$ is a generalized block-cut tree.}
\end{lemma}

\begin{proof}
Suppose vertices $B_i, B_j \in \mathcal{B}$ have edges with $c_p, c_q \in \mathcal{C}$ to form a cycle. But this is absurd, since it means $B_i$ and $B_j$  have two vertices in common, violating \eqref{eq:generalized-blocks}. Thus the generalized block-cut sets $(\mathcal{B},\mathcal{C})$ induce a generalized block-cut tree associated with $G$.
\end{proof}
For these generalized blocks, $G_{\mathcal{BC}}$ corresponding to $\operatorname{bc}(G; \mathcal{B}, \mathcal{C})$ is defined exactly as  in Definition~\ref{def:topology-change-replication}.
\begin{definition}
For generalized block-cut sets $(\mathcal{B}_u, \mathcal{C}_u)$ and $(\mathcal{B}_l, \mathcal{C}_l)$  that satisfy Eq.~\eqref{eq:generalized-blocks}, we can specify an order relation $\succ$ as follows: $(\mathcal{B}_u, \mathcal{C}_u) \succ (\mathcal{B}_l, \mathcal{C}_l)$ if $\mathcal{C}_u \supset \mathcal{C}_l$ and for every $B_u \in \mathcal{B}_u$, there exists a $B_l \in \mathcal{B}_l$ such that $B_u \subset B_l$.
% \begin{gather*}
% (\mathcal{B}_u, \mathcal{C}_u) \succ (\mathcal{B}_l, \mathcal{C}_l) \\
%  \mathrm{if} \; \mathcal{C}_u \supset \mathcal{C}_l, \; \mathrm{and}\\
%  \textrm{for every}\; B_u \in \mathcal{B}_u, \exists B_l \in \mathcal{B}_l \; \textrm{such that} \; B_u \subset B_l.
% \end{gather*}
\end{definition}

\begin{theorem}[Hierarchical Generalized block-cut sets]
\label{thm:generalized-block-cut-sets}
Given a connected graph $G$, there exists  an integer $n$ and hierarchical generalized block-cut sets
$$(\mathcal{B}_0, \mathcal{C}_0) \succ (\mathcal{B}_1, \mathcal{C}_1) \succ \dotsc (\mathcal{B}_n, \mathcal{C}_n)$$
satisfying \eqref{eq:generalized-blocks}
such that $\mathcal{B}_n$ corresponds to Definition~\ref{def:block} and $\mathcal{C}_n$ corresponds to the set of all articulation points of $G$. Moreover, each of the generalized block-cut sets $(\mathcal{B}_k, \mathcal{C}_k)$ corresponds to a generalized block-cut tree $\operatorname{bc}(G; \mathcal{B}_k, \mathcal{C}_k)$ as in Lemma~\ref{lemma:generalized-block-cut-tree} with $\operatorname{bc}(G; \mathcal{B}_n, \mathcal{C}_n)$ corresponding to $\operatorname{bc}(G)$ in Definition~\ref{def:block-cut-tree}.
    
\end{theorem}

\begin{proof}
    The construction proceeds recursively as follows, starting with  a connected network $G$ and the trivial sets $\mathcal{B}_0= \{G\}, \mathcal{C}_0 = \emptyset$ that satisfy \eqref{eq:generalized-blocks}:
\begin{enumerate}
    \item If $G$ is non-separable, the algorithm terminates trivially.
    \edit{Otherwise, designate $B^* = G$ (for convenience) and choose an articulation point $c_{1} \in V(B^*)$.} Initialize $\mathcal{B}_1 = \mathcal{B}_0 \setminus B^*$.
    \item \label{step:2} Let $\{B_{1,i}\;|\; 1\leqslant i \leqslant k_1\}$ for $k_1 \geqslant 2$ denote the connected components of $B^* - c_1$ (Figure ~\ref{fig:blocks}). 
    % we obtain connected components $\{B_{1,i}\}_{1}^{k_1}$  that are sub-graphs of $B^*$ with the positive integer $k_1 \geqslant 2$ (Figure~\ref{fig:blocks}).  
    \item \label{step:3} For each $1 \leqslant i \leqslant k_1$, construct a sub-graph $B_{1,i}^+$ from $B_{1,i}$ by adding vertex $c_1$ and the set of edges $\{(v, c_1) \in E(B^*)\;|\; v \in B_{1,i} \}$ (Figure~\ref{fig:blocks}).
    \item Now  set $\mathcal{B}_1 = \mathcal{B}_1 \cup \{B_{1,i}^+\;|\; 1\leqslant i\leqslant k_1\}$, $\mathcal{C}_1 = \mathcal{C}_0 \cup \{c_1\}$. The sets $\mathcal{B}_1, \mathcal{C}_1$ obey \eqref{eq:generalized-blocks} and moreover, by Lemma~\ref{lemma:generalized-block-cut-tree}, we can associate a tree, $\operatorname{bc}(G; \mathcal{B}_1, \mathcal{C}_1)$ with $G$.
    \item If after $m$ steps, all the blocks in $\mathcal{B}_m$ are non-separable then we have obtained the block-cut tree $\operatorname{bc}(G)$ described in Definition~\ref{def:block-cut-tree}. Otherwise, we can choose a non-separable $B^* \in \mathcal{B}_m$ and a corresponding articulation point $c_{m+1} \in V(B^*)$ and initialize
    $\mathcal{B}_{m+1} = \mathcal{B}_m\setminus B^*$.
    \item Similar to step \ref{step:2}, let $\{B_{m+1,i}\;|\; 1\leqslant i\leqslant k_{m+1}\}$ for some $k_{m+1} \geqslant 2$ denote the connected components of $B^* - c_{m+1}$. Now, construct $\{B_{m+1,i}^+ \;|\; 1\leqslant i\leqslant k_{m+1}\}$  from $B_{m+1,i}$ by adding edges as described earlier in step \ref{step:3}.
    \item  Set $\mathcal{B}_{m+1} = \mathcal{B}_{m+1} \cup \{B_{m+1,i}^+\;|\;1\leqslant i\leqslant k_{m+1}\}$ and $\mathcal{C}_{m+1} = \mathcal{C}_{m} \cup \{c_{m+1}\}$; they obey Eq~\eqref{eq:generalized-blocks} and the  tree $\operatorname{bc}(G; \mathcal{B}_{m+1}, \mathcal{C}_{m+1})$ is now obtained. 
    \item Eventually, for some $m=n$, this recursion will terminate, for every block in $\mathcal{B}_n$ will be non-separable. Then the tree $\operatorname{bc}(G; \mathcal{B}_n, \mathcal{C}_n)$ will coincide with the \emph{block-cut tree} $\operatorname{bc}(G)$ as in Definition~\ref{def:block-cut-tree}.
\end{enumerate}
\end{proof}
\begin{figure}
    \centering
    \includegraphics[scale=0.25]{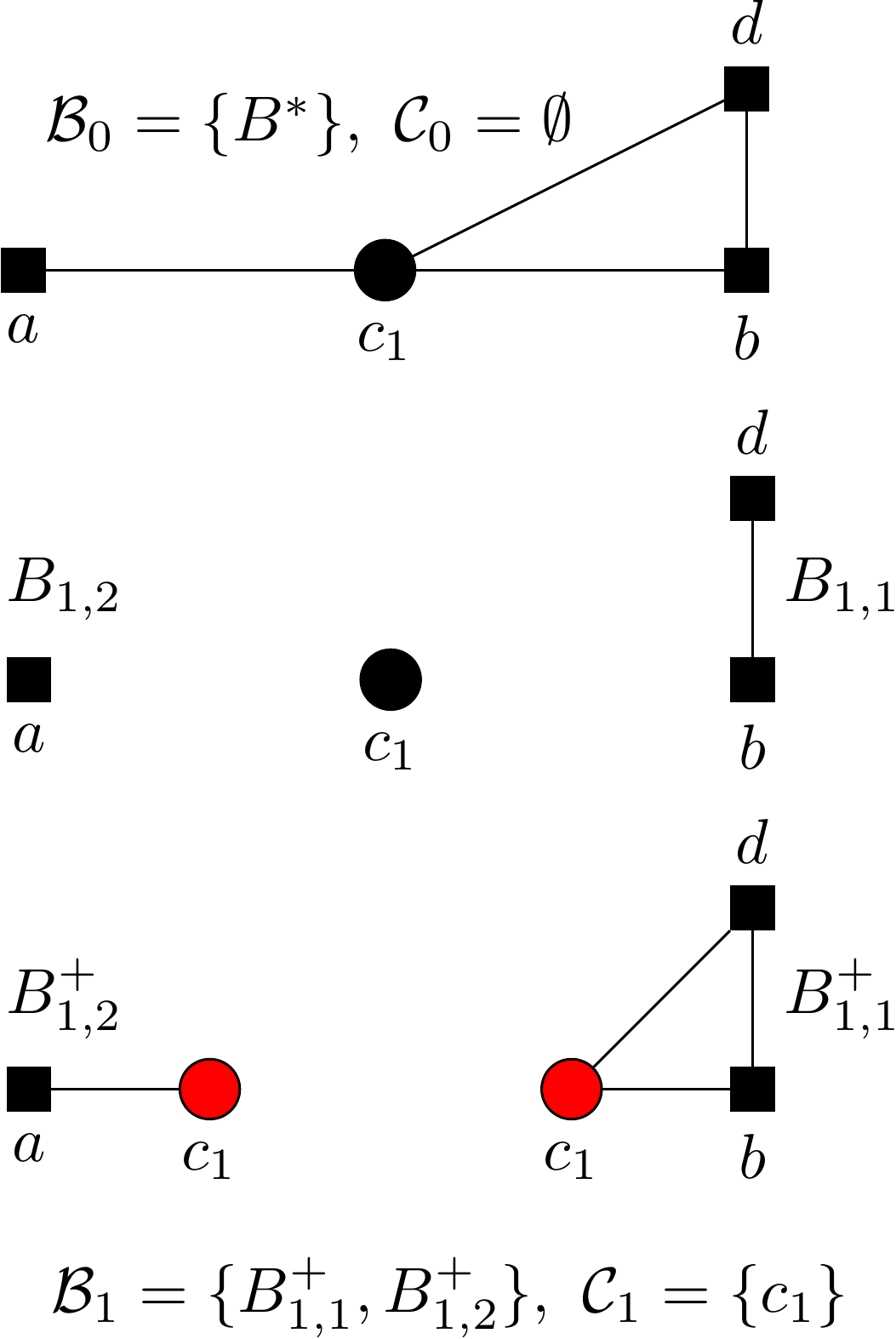}
    \caption{An illustration of Step~\ref{step:2} in the proof of Theorem~\ref{thm:equivalence-generalized-block-cut-system}.}
    \label{fig:blocks}
\end{figure}
The proof of Theorem~\ref{thm:generalized-block-cut-sets} provides a way to use it to control the block sizes in
\begin{corollary}
\label{cor:block-size-control}
Given a connected graph $G$ and a positive integer  $N$, one of the following must hold: (i) A generalized block cut set $(\mathcal{B}, \mathcal{C})$ can be found that satisfies the condition $|V(B)| \leqslant N \; \forall B \in \mathcal{B}$ or (ii) $|V(B)| > N$ for a non-separable $B \in \mathcal{B}$.
\end{corollary}

\subsection{Steady-state network flow in generalized block-cut networks}
In order to exploit a generalized block-cut set, we need to ensure that the addition of (replicated) vertices and edges does not change the solution of the nonlinear system \eqref{eq:NF-pi} defined on  $G$. This will be true if we can show that the equations for the modified network, i.e., \eqref{eq:NF-pi} over $G_{\mathcal{BC}}$ reduce to \eqref{eq:NF-pi} over $G$.  

For ease of notation, let $f^{\operatorname{in}}(i; G)$ and $f^{\operatorname{out}}(i; G)$ denote the total incoming and outgoing flow in $G$ for the junction $i \in V(G)$, i.e., 
\[
f^{\operatorname{in}}(i; G) = \sum_{(j, i) \in E(G)} f_{ji} \text{ and } f^{\operatorname{out}}(i; G) = \sum_{(i, j) \in E(G)} f_{ij}.
\]
% The following theorem proves this formally.

It is sufficient to analyze the case  $\mathcal{C}= \{c_1\}$,  since multiple elements of $\mathcal{C}$ can be considered individually in turn. 
Following the same notation as earlier in Theorem~\ref{thm:generalized-block-cut-sets}, let $\{B_1, B_2, \dotsc, B_m\}$ be the connected components of $G$ when the vertex $c_1$ is removed.
For each $B_i$, the cut-point $c_1$ is replicated as $c_{1i}$ in $G_{\mathcal{BC}}$, with edges of the form $(v, c_1)$ replaced by $(v, c_{1i})$  to form $B_i^+$, and edges of the form $(c_{1i}, c_1)$ added. The replicated vertices are tagged with distinct indices to prevent confusion.
Note that $q^*_{c_{1i}} = 0$ for $i = 1, 2, \dotsc m$. 
Thus $2m$ additional degrees of freedom are introduced for the network flow equations \eqref{eq:NF-pi} defined on $G_{\mathcal{BC}}$ as follows:
% With this idea in mind, consider a single cut-point vertex $c_1 \in \mathcal{C} \subset V(G)$. 
% Equation system for $G_{\mathcal{BC}}$: wlog, we assume $|\mathcal C| = 1$ and $\mathcal C = \{c_1\}$ and that removal of $c_1$ results in connected components $\{B_1, B_2, \dotsc, B_m\}$ be the connected components of $G$.
\begin{subequations}
\begin{flalign}
 & \gamma^{*}_{ij}\pi_i - \pi_j = g^{*}_{ij}(f_{ij}) \quad \forall (i, j) \in \bigcup_{1}^{m}E(B_k^+),  \label{eq:NF-pi-edge-1}\\
  & \pi_{c_{1i}} = \pi_{c_1}, \quad i = 1, 2, \dotsc m, \label{eq:pot-equal} \\ 
& f^{\operatorname{out}}(i; G_{\mathcal{BC}})  - f^{\operatorname{in}}(i; G_{\mathcal{BC}}) = q^{*}_i \quad \forall i \in \bigcup_{1}^{m}V(B_k) \setminus N_s,  \label{eq:NF-pi-balance-nonci}\\
 % & {\sum}_{i=1}^{m} f_{c_1c_{1i}}  = q^{*}_{c_1}  \label{eq:c1} \\ 
 & f^{\operatorname{out}}(c_1; G_{\mathcal{BC}}) = q^*_{c_1}, \label{eq:c1} \\
 % & \underset{v \in V(B_i^+) }{\sum}(f_{vc_{1i}}-f_{c_{1i}v}) - f_{c_1c_{1i}}  = q^{*}_{c_{1i}},\; i=1, 2, \dotsc m  \label{eq:c1i} \\
 & f^{\operatorname{out}}(c_1; B_i^+) - f^{\operatorname{in}}(c_1; B_i^+) + f_{c_{1i} c_1} = q^*_{c_{1i}}, ~i = 1, \dotsc m  \label{eq:c1i} \\
 & \pi_i = \pi^{*}_i \; \forall i \in N_s
\end{flalign}
\label{eq:NF-GBC}
\end{subequations}
The following theorem shows the equivalence of the two systems \eqref{eq:NF-pi} and \eqref{eq:NF-GBC}.
\begin{theorem}
\label{thm:equivalence-generalized-block-cut-system}
Given a connected network $G$ and a generalized block-cut set $(\mathcal{B}, \mathcal{C})$,
the steady-state network flow equations for $G$ and $G_{\mathcal{BC}}$ are equivalent.
\end{theorem}
\begin{proof}
Using \eqref{eq:pot-equal} in \eqref{eq:NF-pi-edge-1} allows  us to consolidate the two equations in the form
\begin{equation}
    \gamma^{*}_{ij}\pi_i - \pi_j = g^{*}_{ij}(f_{ij}) \quad \forall (i, j) \in E(G) \label{eq:NF-pi-edge-G}
\end{equation}
with the caveat that an appropriate $f_{vc_{1i}}$ and $f_{c_{1i}v}$ appear in place of  $f_{vc_{1}}$ and $f_{c_{1}v}$  respectively.
Notice that the left-hand side of \eqref{eq:c1} corresponds to
\begin{equation} 
f^{\operatorname{out}}(c_1; G_{\mathcal{BC}}) = - \sum_{i = 1}^m f_{c_{1i} c_1} \label{eq:c1_out}
\end{equation}
Now, if we sum \eqref{eq:c1} and \eqref{eq:c1i} and use \eqref{eq:c1_out} together with $q_{c_{i1}}^* = 0$ for all $i \in \{1, \dots, m\}$, we get
\begin{equation}
    \sum_{i=1}^{m} \left\{f^{\operatorname{out}}(c_1; B_i^+) - f^{\operatorname{in}}(c_1; B_i^+)\right\} =  q^{*}_{c_1}
    \label{eq:NF-pi-balance-c1}
\end{equation}
Thus, the system \eqref{eq:NF-GBC} reduces to
\begin{subequations}
\begin{flalign}
& \gamma^{*}_{ij}\pi_i - \pi_j = g^{*}_{ij}(f_{ij}) \quad \forall (i, j) \in E(G),  \\
& f^{\operatorname{out}}(i; G_{\mathcal{BC}})  - f^{\operatorname{in}}(i; G_{\mathcal{BC}}) = q^{*}_i \quad \forall i \in \bigcup_{1}^{m}V(B_k) \setminus N_s, \\
& \sum_{i=1}^{m} \left\{f^{\operatorname{out}}(c_1; B_i^+) - f^{\operatorname{in}}(c_1; B_i^+)\right\} =  q^{*}_{c_1} \\ 
& \pi_i = \pi^{*}_i \; \forall i \in N_s. 
\end{flalign}
\end{subequations}
This is precisely the system \eqref{eq:NF-pi} for $G$ wherein  an appropriate $f_{vc_{1i}}$ and $f_{c_{1i}v}$ appear in place of the unknowns $f_{vc_{1}}$ and $f_{c_{1}v}$ respectively. Equivalently, this implies that if we can solve the nonlinear system for $G_{\mathcal{BC}}$, then  the nodal potentials and most of the flows of $G_{\mathcal{BC}}$ correspond to those of $G$ in an obvious way.  The exceptions are the flows of the form $f_{vc_{1i}} \in G_{\mathcal{BC}}$ which correspond to $f_{vc_1} \in G$.
 \end{proof}

\subsection{Hierarchical Solution Algorithm}

For very large networks where solution of the full nonlinear system could be computationally prohibitive, Corollary~\ref{cor:block-size-control} and Theorem~\ref{thm:equivalence-generalized-block-cut-system} allow one to design a solution scheme for a generalized block-cut tree that solves the nonlinear system one block at a time and thus circumvent the computational issues caused by large networks.

One limitation of this scheme is that it requires all slack nodes  belong to the same generalized block, which could be violated in case of multiple slack nodes.
Before we get to the solution scheme, the following useful lemma is presented.

\begin{lemma}
\label{lemma:flow-solve}
    Given a network with $G$ with injections as in \eqref{eq:NF-pi} and a generalized block-cut set $(\mathcal{B}, \mathcal{C})$ such that a generalized-block $B \in \mathcal{B}$ contains all the slack vertices, the flow in each edge of the generalized block-cut tree $\operatorname{bc}(G; \mathcal{B}, \mathcal{C})$ is uniquely determined.
\end{lemma}
\begin{proof}
    The generalized block-cut tree has $n = |\mathcal{B}| + |\mathcal{C}|$ vertices and $n - 1$ edges. By  assigning zero injection to replicated vertices and adding up all the nodal injections in each block, we find that each vertex in the generalized block-cut tree can be assigned an effective injection, except for the vertex representing the block containing slack nodes (since injections at slack vertices are unknown). Denote the agglomerated injection vector by 
    $\tilde{q} \in \mathbb{R}^{(n-1)}$ 
    while the flows in the edges are denoted by $\tilde{f} \in \mathbb{R}^{(n-1)}$.
    \edit{If $\tilde{A}$ be the $(n-1) \times (n-1)$ incidence matrix obtained by ignoring the vertex that represents the block containing slack nodes, then we have the linear system $\tilde{A} \tilde{f} = \tilde{q}$ which has a unique solution since $\tilde{A}$ corresponding to a tree is invertible. Thus the flows $\tilde{f}$ are known. }
\end{proof}

The generalized block-cut tree can be traversed in a finite number of steps starting with any given block (vertex) and this fact forms the basis of   the solution scheme  we now describe:

\begin{theorem}
\label{thm:hierarchical-nonlin-soln}
Given injections and slack potentials as in \eqref{eq:NF-pi}, consider a network $G$ and a generalized block-cut set $(\mathcal{B}, \mathcal{C})$ where all the slack vertices are in a single generalized-block $B^* \in \mathcal{B}$. Then the solution of the steady-state flow system \eqref{eq:NF-pi}  of $G$ may be obtained by solution of the $|\mathcal{B}|$ smaller nonlinear systems  in $G_{\mathcal{BC}}$ in an appropriate sequence. 
\end{theorem}
\begin{proof}
We have already shown (Theorem~\ref{thm:equivalence-generalized-block-cut-system}) that  the solution of the system for $G_{\mathcal{BC}}$ yields the solution corresponding to $G$. This result goes further and decomposes the solution of the large non-linear system for $G_{\mathcal{BC}}$ into smaller non-linear systems, one for each $B \in \mathcal{B}$. The order, however is hierarchical, and  is the crucial ingredient in the proof.

To start, the flow in each edge of the generalized block-cut tree $\operatorname{bc}(G; \mathcal{B}, \mathcal{C})$ is determined as shown  by Lemma~\ref{lemma:flow-solve}. These flows will act as injection boundary conditions for the sub-networks.

As the  slack nodes belong to the block $B^* \in \mathcal{B}$, we term it the \emph{first-level sub-network} 
$$\mathcal{B}^{(1)} = \{B^*\}.$$
For simplicity, we assume that no slack nodes belong to $\mathcal{C}$, so that $\mathcal{B}^{(1)}$ is singleton. The known flows along the edges involving $B^*$ act as injections and the non-linear system \eqref{eq:NF-pi} corresponding to \emph{first-level sub-network} $B^* \in \mathcal{B}^{(1)} $ may now be solved since we assumed a solution exists and by Proposition~\ref{prop:uniqueness} the solution must be unique.

By \eqref{eq:pot-equal}, the computed solution determines the nodal potential for 
\edit{$$\mathcal{C}^{(1)} = \bigcup_{B\in \mathcal{B}^{(1)}}\left(\mathcal{C} \cap V(B)\right),$$} and  sets a slack boundary condition \eqref{eq:NF-pi-slack} for  \emph{second-level sub-networks} 
\edit{$$\mathcal{B}^{(2)} = \left\{B \in \mathcal{B}\setminus \mathcal{B}^{(1)} \; | \; V(B) \cap \mathcal{C}^{(1)} \neq \emptyset \right\}.$$ }

In general, for a positive integer $m\geqslant 1$, if   $\mathcal{B}^{(m)} \neq \emptyset$,  then solution of  the nonlinear system \eqref{eq:NF-pi} for each element of $\mathcal{B}^{(m)}$ determines the nodal potential
\edit{$$\mathcal{C}^{(m)} = \bigcup_{B\in \mathcal{B}^{(m)}}(\mathcal{C} \cap V(B)),$$}
and sets a slack boundary condition \eqref{eq:NF-pi-slack} for  the \emph{$(m+1)$\textsuperscript{th}-level sub-network} defined through
\edit{$$\mathcal{B}^{(m+1)} = \left\{B \in \mathcal{B}\setminus \bigcup_{k \leqslant m}\mathcal{B}^{(k)} \;| \; V(B) \cap \mathcal{C}^{(m)} \neq \emptyset \right\}.$$}

If  $\mathcal{B}^{(m+1)} \neq \emptyset$, the nonlinear system \eqref{eq:NF-pi} is solved for each element of these 
\emph{$(m+1)$\textsuperscript{th}-level sub-networks}. 

For some $m=n$, we will have $\bigcup_{k \leqslant m}\mathcal{B}^{(k)} = \mathcal{B}$, implying that $\mathcal{B}^{(m+1)} = \emptyset$ and so this iterative procedure will terminate. This indicates that  the solution of the nonlinear system \eqref{eq:NF-pi} for $G_{\mathcal{BC}}$ has been obtained by hierarchical traversal of the generalized blocks  in $\mathcal{B}$  and the solution of nonlinear systems associated with these blocks. 
\end{proof}
\begin{remark}
When all slack nodes are in $B^*$, it is possible that there are other blocks containing  a  single slack node. In that case,  $\mathcal{B}^{(1)}$ will not be a singleton but will also  contain  these blocks.  Otherwise, the proof of Theorem~\ref{thm:hierarchical-nonlin-soln} remains unchanged. 
\end{remark}
\begin{remark}
\label{remark:2-node}
The smallest  block consists of two nodes connected by an edge, and hence has three variables (two nodal potentials and one edge flow). 
For such a block, it is trivial to solve the system \eqref{eq:NF-pi} by direct substitution and it illustrates a desirable feature of the algorithm described in Theorem~\ref{thm:hierarchical-nonlin-soln}.
\end{remark}
\section{Results}
\label{sec:results}
The hierarchical algorithm in the previous section was proposed under the assumption that each system of the type \eqref{eq:NF-pi} generated by the sub-networks  had a solution. 
If we restrict attention to gas flow in pipeline networks, then from \cite{ss-soln-existence}, we know that a unique solution to systems of the form \eqref{eq:NF-pi} does exist and hence one can test the algorithm in this context. 

For a few benchmark networks \cite{Schmidt2017}, \cite{birchfield2024structural} the block-cut tree representation was constructed and some key statistics related  to it are reported in Table~\ref{table:stats} \edit{with the aid of the python's NetworkX graph library \cite{networkx}.}
Each network name indicates the size measured through the number of nodes. 
The number of blocks represents the number of sub-networks over which the nonlinear system \eqref{eq:NF-pi} needs to be solved.
The third column indicates the number of minimal blocks in the partition. As these have only two nodes, they can be solved by direct substitution as indicated in Remark~\ref{remark:2-node}, meaning that one does not have to solve these iteratively. 
Thus, for these partitions, iterative methods are not required at all for GasLib-134 (since it is a tree), while GasLib-582 has 11 sub-networks that require some iterative method. 
The last  column reports the size of the biggest sub-network in this partitioned representation of the full nonlinear system and for these cases, this is at most a third of the size of the full network. 
Traversing through the partition as described in Theorem~\ref{thm:hierarchical-nonlin-soln}, a solution to \eqref{eq:NF-pi} was found successfully for each network listed in the table. Also, Fig. \ref{fig:texas-blocks} shows the Texas-2451 network with the sub-networks that have more than two nodes colored.

\begin{table}
\centering
\caption{Key statistics of the block-cut tree representation for some benchmark networks. Note that the 2-node blocks can be solved directly as explained in Remark~\ref{remark:2-node} and the biggest block in each case listed represents the largest sub-network over which the system \eqref{eq:NF-pi} needs to be solved.}
\begin{tabular}{cccc}
\toprule 
\multicolumn{1}{l}{Network}  & \multicolumn{1}{l}{No. of blocks} & \multicolumn{1}{l}{2-node Blocks} & \multicolumn{1}{l}{Max Block size} \\
\midrule
GasLib-40 & 25 &  21 & 11 (27.5\%)  \\
GasLib-134 & 133 & 133 & ----- \\
GasLib-582 & 370 & 359 & 113 (19\%)  \\
Texas-2451 & 1476 & 1470 & 843 (34\%) \\
\bottomrule
% \multicolumn{1}{l}{} & \multicolumn{1}{l}{} & \multicolumn{1}{l}{} & \multicolumn{1}{l}{}
\end{tabular}
\label{table:stats}
\end{table}

\begin{figure}
    \centering
    \includegraphics[scale=1.5]{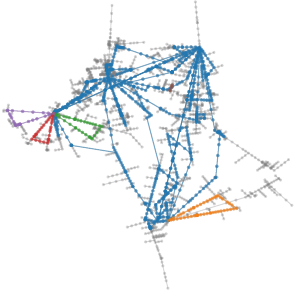}
    \caption{The Texas-2451 network is shown with the blocks that have more than two nodes colored.}
    \label{fig:texas-blocks}
\end{figure}

\section{Conclusion}
Potential-driven, steady-state flow in networks results in a  system of nonlinear equations that depends on the network topology and there is no numerical algorithm that offers guaranteed convergence to the solution in general.
For very large networks where solution of the full nonlinear system could be computationally prohibitive,  a solution scheme was presented that partitions the network into  a generalized
block-cut tree and solves the nonlinear system one block at a time and thus circumvents the computational issues caused by large networks. 
While this work considered the simulation of steady-state flow, it is not clear whether such partitioning confers any advantage in case of transient flow as well. 
Future work could also examine the utility of such partitioning in the context of optimization.

\printbibliography
\vskip 0pt plus -1fil

\end{document}